\documentclass[11pt]{article}%
\usepackage{amssymb}
\usepackage{amsmath}
\usepackage{euler}
\usepackage{amsfonts}
\usepackage{graphicx}%
\providecommand{\U}[1]{\protect\rule{.1in}{.1in}}
\newtheorem{theorem}{Theorem}

\newtheorem{corollary}[theorem]{Corollary}

\newtheorem{proposition}[theorem]{Proposition}

\newenvironment{proof}[1][Proof]{\noindent\textbf{#1.} }{\ \rule{0.5em}{0.5em}}
\begin{document}

\date{}
\title{\textsf{Notes on q-normed Spaces in Constructive Analysis}}
\author{\textsf{Douglas S. Bridges}}
\maketitle

\begin{abstract}%
\noindent
\textsf{What we call q-normed (linear) spaces were introduced into
constructive analysis, under the name pseudonormed spaces, in \cite{Johns}, as
a means of handling spaces, such as }$L_{\infty}$,\textsf{\ in which not all
elements are constructively normable. We prove a number of q-normed-space
analogues/generalisations of standard theorems in the constructive analysis of
normed linear spaces, and give examples showing that two natural analogues are
essentially nonconstructive.}

\end{abstract}

%

\setcounter{secnumdepth}{0}%
\normalfont\sf

\section{Introduction}

Working within the framework of Errett Bishop's constructive analysis
(\textbf{BISH}),\footnote{%
\normalfont\sf
See, for example, Bishop's original work \textsf{\cite{Bishop}, }the later
reference \cite{BVtech}, or Ishihara's excellent exposition in Chapters 8 and
9 of \cite{Handbook}.} in this article we present some constructive work on
q-normed spaces,\footnote{%
\normalfont\sf
Johns \cite{Johns} used the name \emph{pseudonormed space;} in both
\textsf{\cite[Chap. 7, Sect. 5]{BB} and \cite{dsbconvex} we used
}\emph{quasinormed space}. However, pseudonorm and quasinorm are applied
classically to notions different from ours, so I have decided to adopt an
entirely new, safer nomenclature.} a constructive generalisation of normed
linear spaces that was introduced by Johns in his PhD. thesis \cite{Johns}.
For the reader's benefit we begin by gathering together some fundamental definitions.

We define a \textbf{q-normed space} to be a pair $\left(  X,\left(  \left\Vert
\ \right\Vert _{i}\right)  _{i\in I}\right)  $ comprising a real\footnote{%
\normalfont\sf
For simplicity, we consider only real linear spaces here.} linear space $X$
equipped with a \textbf{q-norm}---that is, a family $(\left\Vert \ \right\Vert
_{i})_{i\in I}$ of seminorms such that

\begin{quote}
(*) \ \ for each $x\in X$ the set $\left\{  \left\Vert x\right\Vert _{i}:i\in
I\right\}  $ is bounded in $\mathbf{R}$.
\end{quote}

%

\noindent
The inequality and equality on this q-normed space are defined by

\begin{description}
\item[ ] $x\neq y$ if and only if there exists $i\in I$ such that $\left\Vert
x-y\right\Vert _{i}>0$, and

\item[ ] $x=y$ if and only if $\left\Vert x-y\right\Vert _{i}=0$ for all $i\in
I$.
\end{description}

%

\noindent
Thus a q-normed space is a special type of locally convex space,\footnote{%
\normalfont\sf
For more on locally convex spaces in BISH, see \cite[Ch. 9, Sect. 5]{Bishop}
or, for a more up-to-date exposition, \cite{dsbconvex}.} one in which (*) holds.

An element $x$ of $X$ is \textbf{normable} with respect to the q-norm
$(\left\Vert \ \right\Vert _{i})_{i\in I}$ if the \textbf{norm}%
\begin{equation}
\left\Vert x\right\Vert \equiv\sup\left\{  \left\Vert x\right\Vert _{i}:i\in
I\right\}  \label{ab1}%
\end{equation}
of $X$ exists. If every element of $X$ is normable, then we can regard $X$ as
a normed (linear) space relative to the norm defined at (\ref{ab1}).

A simple example of a q-normed space is a normed space $X$, where we take
$I=\left\{  1\right\}  $ and $\left\Vert \ \right\Vert _{1}$ equal to the
given norm on $X$. Our various definitions for properties and elements of
q-normed spaces are chosen so that for a normed space they coincide with their
counterparts in the standard theory of normed spaces. A more interesting
example of a q-normed space is given by the Lebesgue space $L_{\infty}\ $of
essentially bounded subsets of $\mathbf{R}$ with respect to Lebesgue measure,
in which example the classically valid statement

\begin{quote}
\emph{Every element of }$L_{\infty}$\emph{\ is normable}
\end{quote}

%

\noindent
is not constructively provable; for details see \cite[pages 346--350]{BB}.

For each element $a\ $of the q-normed space $\left(  X,\left(  \left\Vert
\ \right\Vert _{i}\right)  _{i\in I}\right)  ,$ and each $r>0$, the
\textbf{open and closed balls} with centre $a$ and radius $r$ are,
respectively,%
\begin{align*}
B(a,r)  &  \equiv\left\{  x\in X:\exists_{r^{\prime}}(0<r^{\prime}%
<r\wedge\forall_{i\in I}(\left\Vert x-a\right\Vert _{i}<r)\right\}  ,\\
\overline{B}(a,r)  &  \equiv\left\{  x\in X:\left\Vert x-a\right\Vert _{i}\leq
r\text{ for all }i\in I\right\}  .
\end{align*}
The \textbf{unit ball} of $X$ is $B_{X}\equiv\overline{B}(0,1)$. These
definitions of open and closed balls reduce to the standard ones when $X$ is a
normed space.

The \textbf{q-norm topology} $\tau_{qX}\ $on $X$ consists of all subsets $S$
of $X$ such that for each $x\in S$ there exists $r>0$ such that $B(x,r)\subset
S$. The elements of $\tau_{qX}$ are said to be $\tau_{qX}$\textbf{-open} (or,
if the context is clear, just open). The $\tau_{qX}$\textbf{-closure} of a
subset $S$ of $X$ is the set $S^{c}$ consisting of all $x\in X$ such that
$B(x,r)$ intersects $S$ for each $r>0$. We say that $S $ is $\tau_{qX}%
$\textbf{-closed} (or, again if the context is clear, just closed) in $X$ if
$S=S^{c}$. We say that a subset $T$ of $X$ is \textbf{dense in }$S$ if
$T^{c}=S$, and that $S$ is \textbf{separable }if it has a countable sense subset.

Let $Y$ be a q-normed space, with q-norm $(\left\Vert \ \right\Vert
_{j})_{j\in J}$. A mapping $f:S\rightarrow Y$, where $S\subset X$, is

\begin{itemize}
\item \textbf{continuous at the point} $a\in S\ $if for each $\varepsilon>0$
there exists $\delta>0$ such that if $x\in S\ $and $\left\Vert
f(x)-f(a)\right\Vert _{j}>\varepsilon$ for some $j\in J$, then $\left\Vert
x-a\right\Vert _{i}>\delta$ for some $i\in I$;

\item \textbf{uniformly continuous on }$S$ if for each $\varepsilon>0$ there
exists $\delta>0$ such that if $x,x^{\prime}\in S$ and $\left\Vert
f(x)-f(x^{\prime})\right\Vert _{j}>\varepsilon$ for some $j\in J$, then
$\left\Vert x-x^{\prime}\right\Vert _{i}>\delta$ for some $i\in I$. We then
call $\delta$ a \textbf{continuity associate} for $\varepsilon$ and $f$;

\item \textbf{bi-uniformly continuous} if it is injective and uniformly
continuous, and $f^{-1}$ is uniformly continuous on $f(S)$.
\end{itemize}

\begin{proposition}
\label{mar31p1}Let $\left(  X,\left(  \left\Vert \ \right\Vert _{i}\right)
_{i\in I}\right)  $ be a q-normed space and $Y$ a normed space. Let $a\in
S\subset X$ and let $f:S\rightarrow Y$ be a mapping that is continuous at $a$
with respect to the locally convex structures given by the q-norms. Then $f$
is continuous at $a$ in with respect to the q-norms.
\end{proposition}

\begin{proof}
According to \cite{dsbconvex}, the continuity assumption for $f$ means that
for each each $\varepsilon>0$ there exist a finitely enumerable subset
$F\equiv\left\{  i_{1},\ldots,i_{N}\right\}  $ of $I$ and $\delta_{0}>0$ such
that if $x\in S$ and $\sum_{i\in F}\left\Vert x-a\right\Vert _{i}<\delta_{0}$,
then $\left\Vert f(x)-f(a)\right\Vert <\varepsilon$. Let $\delta=\delta
_{0}/2N>0$. If $x\in S$ and $\left\Vert f(x)-f(a)\right\Vert >\varepsilon$,
then%
\[
\sum_{i\in F}(\left\Vert x-a\right\Vert _{i}-\delta)=\sum_{i\in F}\left\Vert
x-a\right\Vert _{i}-N\delta>\sum_{i\in F}\left\Vert x-a\right\Vert _{i}%
-\delta_{0}\geq0,
\]
so there exists $\iota\in F$ such that $\left\Vert x-a\right\Vert _{\iota
}-\delta>0$ and therefore $\left\Vert x-a\right\Vert _{\iota}>\delta$.
\end{proof}

\begin{proposition}
\label{apr08p1}Let $\left(  X,\left(  \left\Vert \ \right\Vert _{i}\right)
_{i\in I}\right)  $ be a q-normed space and $Y$ a normed space. Let $S$ be a
subset of $X$, and $f:S\rightarrow Y$ a mapping that is uniformly continuous
on $S$ with respect to the locally convex structures given by the q-norms..
Then $f$ is uniformly continuous on $S\ $with respect to the q-norms.
\end{proposition}

\begin{proof}
Similar to that of the previous proposition.
\end{proof}

%

\medskip

A sequence $\left(  x_{n}\right)  _{n\geq1}$ of elements of the q-normed space
$X$

\begin{itemize}
\item \textbf{converges} to the (perforce unique) \textbf{limit} $x\in X$ if
for each $\varepsilon>0$ there exists a positive integer $N$ such that
$\left\Vert x-x_{n}\right\Vert _{i}<\varepsilon$ for all $n\geq N$ and all
$i\in I$;

\item is a \textbf{Cauchy sequence} if for each $\varepsilon>0$ there exists a
positive integer $N$ such that $\left\Vert x_{m}-x_{n}\right\Vert
_{i}<\varepsilon$ for all $m,n\geq N$ and all $i\in I$.
\end{itemize}

%

\noindent
We say that $X$ is \textbf{complete} if every Cauchy sequence in $X$ converges
to a limit in $X$.

Let $X$ and $Y$ be q-normed spaces. Leaving the straightforward proofs to the
reader, we state some facts about a mapping $f$ of a subset $S$ of $X$ into
$Y$.

\begin{itemize}
\item[--] If $f$ is continuous at $a\in S$ and $\left(  x_{n}\right)
_{n\geq1}$ is a sequence in $S$ converging to $a$, then $\left(
f(x_{n})\right)  _{n\geq1}$ converges to $f(a)$.

\item[--] If $f$ is uniformly continuous on $S$ and $\left(  x_{n}\right)
_{n\geq1}$ is a Cauchy sequence in $X$, then $\left(  f(x_{n})\right)
_{n\geq1}$ is a Cauchy sequence in $f(S)$.

\item[--] If $f\ $is a bi-uniformly continuous injection of $S$ into $Y$, and
$S$ is complete, then $f(S)$ is complete.
\end{itemize}

%

\noindent
We shall invoke these facts without comment in future.%

\medskip

Turning to linearity, we say that a linear mapping $T$ of $X$ into $Y$ is
\textbf{bounded} if there exists $c>0~$( a \textbf{bound} for $T$) with the
property: for each $x\in X$, each $j\in J,$ and each $\varepsilon>0$ there
exists $i\in I$ such that $c\left\Vert x\right\Vert _{i}>\left\Vert
Tx\right\Vert _{j}-\varepsilon$. For a normable element $x$ of $X$, this
condition holds if and only if $c\left\Vert x\right\Vert \geq\left\Vert
Tx\right\Vert _{j}$ for each $j$.

\begin{proposition}
\label{feb13p1}Let $\left(  X,\left(  p_{i}\right)  _{i\in I}\right)  $ and
$\left(  Y,\left(  q_{j}\right)  _{j\in J}\right)  $ be q-normed spaces, and
$u:X\rightarrow Y$ a linear mapping. Then the following are equivalent conditions.

\begin{itemize}
\item[\emph{(i)}] $u$ is continuous at $0$.

\item[\emph{(ii)}] $u$ is continuous on $X$.

\item[\emph{(iii)}] $u$ is uniformly continuous on $X$.

\item[\emph{(iv)}] $u$ is bounded.
\end{itemize}
\end{proposition}

\begin{proof}
Assume that $u$ is continuous at $0$, and let $\varepsilon>0$. Then there
exists $\delta>0$ such that if $x\in X\ $and $\left\Vert u(x)\right\Vert
_{j}>\varepsilon$ for some $j\in J$, then $\left\Vert x\right\Vert _{i}%
>\delta$ for some $i\in I$. So if $x,x^{\prime}\in X$ and $\left\Vert
u(x)-u(x^{\prime})\right\Vert _{j}>\varepsilon$, then $\left\Vert
u(x-x^{\prime})\right\Vert _{j}$, so there exists $i\in I$ such that
$\left\Vert x-x^{\prime}\right\Vert _{i}>\delta$. Thus (iii), and clearly also
(ii), holds. It is trivial that (iii) implies both (ii) and (i). Finally, it
is proved in \cite[Chap. 7, Prop. 5.6]{BB} that (iii) and (iv) are equivalent.
\end{proof}

%

\medskip

We define the \textbf{standard q-norm on the set} $\mathcal{B}(X,Y)$ of
bounded linear mappings from $X$ into $Y$ to be the family of seminorms%
\[
\left\Vert \ \right\Vert _{x,j}:T\rightsquigarrow\left\Vert Tx\right\Vert
_{j}\ \ \left(  x\in B_{X},~j\in J\right)  .
\]
Interpreted classically, the normability of $T$ is equivalent to the existence
of%
\[
\sup_{x\in B_{X}}\sup\left\{  \left\Vert Tx\right\Vert _{j}:j\in J\right\}
=\sup_{x\in B_{X}}\left\Vert Tx\right\Vert ,
\]
which is precisely the classical norm of $T$.

We shall be particularly concerned with the \textbf{dual q-normed space}
$X^{\ast}\equiv\mathcal{B}(X,\mathbf{C})\ $of $X$,$\ $and the \textbf{second
dual q-normed space} $X^{\ast\ast}=(X^{\ast})^{\ast}$ of $X$. For each $x\in
X\ $we denote by $\widehat{x}\ $the linear mapping $u\rightsquigarrow u(x)$ on
$X^{\ast}$. Then $\widehat{x}\in X^{\ast\ast}$, since for each $u\in X^{\ast}$
and each $\varepsilon>0$ we have
\[
\left\vert \widehat{x}(u)\right\vert -\varepsilon=\left\vert u(x)\right\vert
-\varepsilon=\left\Vert u\right\Vert _{x}-\varepsilon<\left\Vert u\right\Vert
_{x}.
\]
For each $S\subset X$ we write $\widehat{S}$ for the image of $S$ under the
\textbf{natural embedding }$x\rightsquigarrow\widehat{x}$ of $X$ into
$X^{\ast\ast}$.

\begin{proposition}
\label{feb13p2}Let $\left(  X,(\left\Vert \ \right\Vert _{i})_{_{i\in I}%
}\right)  $ be a q-normed space, and $x\in X$. If $\left\Vert x\right\Vert
_{i}\leq c$ for all $i\in I$, then $\left\Vert \widehat{x}\right\Vert _{u}\leq
c$ for all $u\in B_{X}^{\ast}$. In particular, if $x$ is normable, then
$\left\Vert \widehat{x}\right\Vert _{u}\leq\left\Vert x\right\Vert $ for all
$u\in B_{X^{\ast}}$.
\end{proposition}

\begin{proof}
For each $u\in B_{X}^{\ast}$ and each $\varepsilon>0$, since $1\ $is a bound
for $u$, there exists $i\in I$ such that $\left\Vert x\right\Vert
_{i}>\left\vert u(x)\right\vert -\varepsilon=\left\Vert \widehat{x}\right\Vert
_{u}-\varepsilon$. If $\left\Vert x\right\Vert _{i}\leq c$ for all $i\in I$,
then $\left\Vert \widehat{x}\right\Vert _{u}-\varepsilon\leq c$. Since
$u,\varepsilon$ are arbitrary, we have $\left\Vert \widehat{x}\right\Vert
_{u}\leq c$ for all $u\in B_{X}^{\ast}$. The rest of the proposition follows immediately.
\end{proof}

%

\medskip

From now on, when we speak of $X$ and $Y$ as q-normed spaces, we assume that
their q-norms are $\left(  \left\Vert \ \right\Vert _{i}\right)  _{i\in I}$
and $\left(  \left\Vert \ \right\Vert _{j}\right)  _{j\in J}$ respectively.
Note that the seminorms $\left(  \left\Vert \ \right\Vert _{x}\right)  _{x\in
B_{X}}$ making $X^{\ast}$ a q-normed space define the same topology as do the
seminorms $\left\Vert \ \right\Vert _{x}$\ $(x\in X)$ that define the standard
locally convex structure on $X^{\ast}$.

How should we define the locatedness of a subset $S$ the q-normed space $X$?
We need to translate $\inf\left\{  \left\Vert x-s\right\Vert :s\in S\right\}
$ from the case where $X$ is a normed space and every element is normable. A
first translation is%
\[
\inf\left\{  \sup_{i\in I}\left\Vert x-s\right\Vert _{i}:s\in S\right\}  ,
\]
which exists if and only if for all $\alpha,\beta\in\mathbf{R}$ with
$\alpha<\beta$, either $\sup_{i\in I}\left\Vert x-s\right\Vert _{i}>\alpha$
for all $s\in S$, or else $\sup_{i\in I}\left\Vert x-s\right\Vert _{i}<\beta$
for some $s\in S$. Unpacking this, we say that the subset $S$ of the q-normed
space $X$ is \textbf{located} in $X$ if for each $x\in X$ and for all
$\alpha,\beta\in\mathbf{R}$ with $\alpha<\beta$,

\begin{itemize}
\item either for each $s\in S$ there exists $i\in I$ such that $\left\Vert
x-s\right\Vert _{i}>\alpha$

\item or else there exists $s\in S$ such that $\left\Vert x-s\right\Vert
_{i}<\beta$ for all $i\in I$.
\end{itemize}

%

\noindent
It is easily seen, using this definition and the constructive
least-upper-bound principle \cite[2.1.18]{BVtech}, that a singleton subset
$\left\{  a\right\}  $ of $X$ is located if and only if $x-a$ is normable for
each $x\in X$.

\begin{proposition}
\label{mar03p2}The following are equivalent conditions on a q-normed space
$X.$

\begin{enumerate}
\item[\emph{(i)}] $\left\{  0\right\}  $ is located in $X.$

\item[\emph{(ii)}] Every element of $X$ is normable.

\item[\emph{(iii)}] Every singleton subset of $X$ is located.

\item[\emph{(iv)}] Every finitely enumerable subset of $X$ is located.
\end{enumerate}
\end{proposition}

\begin{proof}
The remark immediately preceding this proposition shows that $\left\{
0\right\}  $ is located if and only if every element of $X$ is normable. In
that case, for all $x,y\in X$ the element $\left\Vert x-y\right\Vert $ is
normable and therefore, by the same remark, $\left\{  y\right\}  $ is located.
Thus (i) $\Rightarrow\ $(ii)\ $\Rightarrow\ $(iii). If (iii) holds and $x_{k}$
$(1\leq k\leq n)$ belong to $X$, let $\alpha,\beta\in\mathbf{R}$ with
$\alpha<\beta$, and consider any $x\in X$. For each $k$, since $\left\{
x_{k}\right\}  $ is located, either there exists $i\in I$ such that
$\left\Vert x-x_{k}\right\Vert _{i}>\alpha$ or else $\left\Vert x-x_{k}%
\right\Vert _{i}<\beta$ for all $i\in I$. Hence either for each $k\leq n$
there exists $i\in I$ such that $\left\Vert x-x_{k}\right\Vert _{i}>\alpha$ or
else there exists $k\leq n$ such that $\left\Vert x-x_{k}\right\Vert
_{i}<\beta$ for all $i\in I$. It now follows that $\left\{  x_{1},\ldots
,x_{n}\right\}  $ is located in $X$; whence (iii) $\Rightarrow~$(iv). Finally,
it is trivial that (iv) $\Rightarrow$ (i).
\end{proof}

\begin{proposition}
\label{feb24p1}Let $X$\ be a q-normed space, and $u$ a nonzero element of
$X^{\ast}$. Then $u$\ is normable if and only if $\ker u$\ is located in $X$.
\end{proposition}

\begin{proof}
First observe that for each $t>0,$%
\begin{align}
\exists_{y\in\ker u}\forall_{i\in I}\left(  \left\Vert a-y\right\Vert
_{i}<t\right)   &  \Leftrightarrow\exists_{z\in X}(\forall_{i\in I}(\left\Vert
z\right\Vert _{i}<1)\wedge u(a-tz)=0)\label{02a}\\
&  \Leftrightarrow\text{ }u(a)\in tu\left\{  z\in X:\forall_{i\in
I}(\left\Vert z\right\Vert _{i}<1)\right\}  .\nonumber
\end{align}
Suppose that $\ker u$ is located in $X$. Since $u$ is nonzero, there exists
$x_{0}\in X$ such that $u(x_{0})=1$. Let $0<\alpha<\beta$ and $s=\frac{1}%
{2}(\alpha+\beta)$. By the locatedness of $\ker u$, either (i) for each
$y\in\ker u$ there exists $i\in I$ such that $\left\Vert x_{0}-y\right\Vert
>1/\beta$, or else (ii) there exists $y^{\prime}\in\ker u$ such that
$\left\Vert x_{0}-y^{\prime}\right\Vert _{i}<1/s$ for all $i\in I$. Consider
case (i) and any $x\in B_{X}$. Either $\left\vert u(x)\right\vert <\beta$ or
$u(x)\neq0$. In the latter event, $x_{0}-\frac{1}{u(x)}x\in\ker u$, so there
exists $i\in I$ such that%
\[
\frac{1}{\left\vert u(x)\right\vert }\geq\frac{1}{\left\vert u(x)\right\vert
}\left\Vert x\right\Vert _{i}=\left\Vert \frac{1}{u(x)}x\right\Vert
_{i}=\left\Vert x_{0}-\left(  x_{0}-\frac{1}{u(x)}x\right)  \right\Vert
_{i}>\frac{1}{\beta}%
\]
and therefore $\left\vert u(x)\right\vert <\beta$. Thus if (i) holds, then
$\left\vert u(x)\right\vert <\beta$ for all $x\in B_{X}$. On the other hand,
if $y^{\prime}\in\ker u$ and $\left\Vert x_{0}-y^{\prime}\right\Vert _{i}<1/s$
for all $i\in I$, then $z\equiv s(x_{0}-y^{\prime})\in B_{X}$ and
$u(z)=su(x_{0})=s>\alpha$. Since $\alpha,\beta$ are arbitrary, it follows from
the constructive least-upper-bound principle \cite[2.1.18]{BVtech} that
$\left\Vert u\right\Vert $ exists.

Suppose, conversely, that $u$ is normable. Since $u$ is nonzero, $\left\Vert
u\right\Vert >0$. Consider any $a\in X$ and any real numbers $\alpha,\beta$
with $\alpha<\beta$, and let $\delta=\frac{1}{2}(\beta-\alpha)$. Either
$\left\vert u(a)\right\vert >(\alpha+\delta)\left\Vert u\right\Vert $ or
$\left\vert u(a)\right\vert <\beta\left\Vert u\right\Vert $. In the first
case, for each $y\in\ker u$ there exists $i\in I$ such that%
\[
\left\Vert u\right\Vert \left\Vert a-y\right\Vert _{i}>\left\vert
u(a-y)\right\vert -\delta\left\Vert u\right\Vert =\left\vert u(a)\right\vert
-\delta\left\Vert u\right\Vert
\]
and therefore%
\[
\left\Vert a-y\right\Vert _{i}>\frac{\left\vert u(a)\right\vert }{\left\Vert
u\right\Vert }-\delta>\alpha.
\]
In the second case, there exists $\varepsilon$ such that $0<\varepsilon
<\left\Vert u\right\Vert $ and $\left\vert u(a)\right\vert <\beta\left(
\left\Vert u\right\Vert -\varepsilon\right)  $. Choose $x\in B_{X}$ such that
$u(x)>\left\Vert u\right\Vert -\varepsilon$. Then%
\[
z=a-\frac{u(a)}{u(x)}x\in\ker u
\]
and for each $i\in I$,%
\[
\left\Vert a-z\right\Vert _{i}=\frac{\left\vert u(a)\right\vert }%
{u(x)}\left\Vert x\right\Vert _{i}\leq\frac{\left\vert u(a)\right\vert }%
{u(x)}<\frac{\left\vert u(a)\right\vert }{\left\Vert u\right\Vert
-\varepsilon}<\beta.
\]
Since $\alpha,\beta$ are arbitrary, it follows that $\ker u$ is located in $X$.
\end{proof}

%

\medskip

We call the q-normed space $X$ \textbf{uniformly convex }if for each
$\varepsilon>0$ there exists $\delta$ with $0<\delta<1$ such that if $x,y\in
B_{X}$ and $\left\Vert x-y\right\Vert _{i}>\varepsilon$ for some $i\in I$,
then $\left\Vert \frac{1}{2}\left(  x+y\right)  \right\Vert _{i}<1-\delta$ for
all $i\in I$. We refer to $\delta$ as a \textbf{uc-associate number} for
$\varepsilon$.

\begin{proposition}
\label{feb28p2}Let $u$ be a nonzero normable linear functional on a complete,
uniformly convex q-normed space $\left(  X,\left(  \left\Vert \ \right\Vert
_{i}\right)  _{i\in I}\right)  $. Let $\varepsilon>0$ and let $\delta<1$ be a
uc-associate number for $\varepsilon/2$. If $x,y\in B_{X},$%
\begin{equation}
\left\Vert u\right\Vert -u(x)<\frac{\delta}{3}\left\Vert u\right\Vert \text{
and }\left\Vert u\right\Vert -u(y)<\frac{\delta}{3}\left\Vert u\right\Vert ,
\label{6a1}%
\end{equation}
then $\left\Vert x-y\right\Vert <\varepsilon$ for all $i\in I$.
\end{proposition}

\begin{proof}
Consider $x,y\in B_{X}$ such that (\ref{6a1}) holds; then $u(x)>\left(
1-\frac{\delta}{3}\right)  \left\Vert u\right\Vert >$ $0\ $and likewise
$u(y)>0$. Since $\left\Vert u\right\Vert $ is a bound for $u$, there exists
$i_{0}\in I$ such that%
\begin{align*}
\left\Vert u\right\Vert \left\Vert \frac{1}{2}\left(  x+y\right)  \right\Vert
_{i_{0}} &  >\left\vert \frac{1}{2}u(x+y)\right\vert -\frac{\delta}%
{3}\left\Vert u\right\Vert \\
&  =\left\vert u(x)-\frac{1}{2}u(x-y)\right\vert -\frac{\delta}{3}\left\Vert
u\right\Vert \\
&  \geq u(x)-\frac{1}{2}\left\vert u(x-y)\right\vert -\frac{\delta}%
{3}\left\Vert u\right\Vert \\
&  =u(x)-\frac{1}{2}\left\vert u(x)-u(y)\right\vert -\frac{\delta}%
{3}\left\Vert u\right\Vert \\
&  =u(x)-\frac{1}{2}\left[  \left\Vert u\right\Vert -u(y)-\left(  \left\Vert
u\right\Vert -u(x)\right)  \right]  \\
&  \geq u(x)-\frac{1}{2}\left[  \left\vert \left\Vert u\right\Vert
-u(y)\right\vert +\left\vert \left\Vert u\right\Vert -u(x)\right\vert \right]
\\
&  \geq u(x)-\frac{1}{2}(\left\Vert u\right\Vert -u(y))-\frac{1}{2}\left(
\left\Vert u\right\Vert -u(x)\right)  -\frac{\delta}{3}\left\Vert u\right\Vert
\\
&  >\left(  1-\frac{\delta}{3}\right)  \left\Vert u\right\Vert -\frac{\delta
}{6}\left\Vert u\right\Vert -\frac{\delta}{6}\left\Vert u\right\Vert
-\frac{\delta}{3}\left\Vert u\right\Vert \\
&  =\left(  1-\delta\right)  \left\Vert u\right\Vert .\text{ }%
\end{align*}
Hence $\left\Vert \frac{1}{2}(x+y)\right\Vert _{i_{0}}>1-\delta$ and therefore
there cannot exist $i\in I$ with $\left\Vert x-y\right\Vert _{i}%
>\varepsilon/2$. Thus $\left\Vert x-y\right\Vert _{i}<\varepsilon$ for all
$i\in I$.
\end{proof}

%

\bigskip

In the special case of a normed space $X$, the next proposition strengthens
the uniqueness conclusion in \cite[2.3.7]{BVtech}.

\begin{proposition}
\label{feb24p2}If $u$ is a nonzero normable linear functional on a complete,
uniformly convex q-normed space $\left(  X,\left(  \left\Vert \ \right\Vert
_{i}\right)  _{i\in I}\right)  $, then there exists a normable element
$x_{\infty}\in B_{X}$, with norm $1$, such that $u(x_{\infty})=\left\Vert
u\right\Vert $. Moreover, $x_{\infty}$ is \textbf{strongly unique}, in the
sense that if $x\in B_{X}$ and $x\neq x_{\infty}$, then $u(x)<\left\Vert
u\right\Vert $.
\end{proposition}

\begin{proof}
Construct a sequence $\left(  x_{n}\right)  _{n\geq1}$of elements of $B_{X}$
such that $u(x_{n})\rightarrow\left\Vert u\right\Vert $ as $n\rightarrow
\infty$. Let $\varepsilon>0$, and let $\delta<3$ be a uc-associate number for
$\varepsilon/2$. There exists $N$ such that $\left\Vert u\right\Vert
-u(x_{n})<\frac{\delta}{3}\left\Vert u\right\Vert $ for all $n\geq N$. By
Proposition \ref{feb28p2}, $\left\Vert x_{m}-x_{n}\right\Vert _{i}%
<\varepsilon$ for all $m,n\geq N$ and all $i\in I$. Since $\varepsilon>0$ is
arbitrary, we see that $\left(  x_{n}\right)  _{n\geq1}\ $is a Cauchy sequence
in the complete q-normed space $X$ and therefore converges to a limit
$x_{\infty}\in X$. Then $\left\Vert x_{\infty}\right\Vert _{i}=\lim
_{n\rightarrow\infty}\left\Vert x_{n}\right\Vert _{i}\leq1$ for each $i\in X$,
so $x_{\infty}\in B_{X}$. Moreover, since $u$ is continuous, $u(x_{\infty
})=\lim_{n\rightarrow\infty}u(x_{n})=\left\Vert u\right\Vert $. Finally, to
prove the strong uniqueness of $x_{\infty}$, let $x\in B_{X}$ and $x\neq
x_{\infty}$, choose $j\in I$ such that $\left\Vert x-x_{\infty}\right\Vert
_{j}>0$, and let $\alpha$ be a uc-associate number for $\frac{1}{2}\left\Vert
x-x_{\infty}\right\Vert _{j}$. By Proposition \ref{feb28p2}, if $\left\Vert
u\right\Vert -u(x)<\frac{\alpha}{3}\left\Vert u\right\Vert $, then $\left\Vert
x-x_{\infty}\right\Vert _{i}<\left\Vert x-x_{\infty}\right\Vert _{j}$ for all
$i$, which is plainly absurd. Hence $\left\Vert u\right\Vert -u(x)>0$, so
$u(x)<\left\Vert u\right\Vert $.
\end{proof}

%

\medskip

We say that the q-normed space $X$ is

\begin{itemize}
\item \textbf{reflexive} if for each normable element $\phi$ of $X^{\ast\ast}$
there exists a (perforce unique) $x\in X$ such that $\phi=\widehat{x}$;

\item \textbf{pliant\ }if for each\textbf{\ }$x\in X$, each $i\in I$, and each
$\varepsilon>0$, there exists $u\in B_{X^{\ast}}$ such that $u(x)>\left\Vert
x\right\Vert _{i}-\varepsilon$.
\end{itemize}

%

\noindent
Classically, every normed space is pliant. Constructively, a nontrivial normed
space is pliant if it is either separable \cite[5.3.1]{BVtech} or else has
G\^{a}teaux differentiable norm (see \cite[Lemma 1]{Ishihara} or
\cite[5.3.6]{BVtech}). In particular, every Hilbert space, separable or
otherwise, has G\^{a}teaux differentiable norm and so is reflexive.

Perhaps the most interesting result about uniformly convex q-normed spaces is
the constructive \textbf{Milman-Pettis theorem}.

\begin{theorem}
\label{mar05t1}Every complete, pliant, uniformly convex q-normed space $X$ is
reflexive, and the natural embedding into $X^{\ast\ast}$ is a norm-preserving
bijection of the set of normable elements of $X$ onto the set of normable
elements of $X^{\ast\ast}$ \emph{\cite[Theorem 3]{dsbconvex}}.
\end{theorem}

\section{Total boundedness}

Let $S$ \ be a subset of our q-normed space $\left(  X,\left(  \left\Vert
\ \right\Vert _{i}\right)  _{i\in I}\right)  $. Given $\varepsilon>0$, we call
a set $A\subset S$ an $\varepsilon$\textbf{-approximation} to $S$ if for each
$x\in S$ there exists $x^{\prime}\in A$ such that $\left\Vert x-x^{\prime
}\right\Vert _{i}<\varepsilon$ for all $i\in I$. We say that $S$ is

\begin{itemize}
\item \textbf{bounded }if there exists $r>0$ such that $\left\Vert
x\right\Vert _{i}<r$ for all $x\in S$ and all $i\in I$;

\item \textbf{totally bounded }if for each $\varepsilon>0$ there exists a
finitely enumerable $\varepsilon$-approximation to $S$.
\end{itemize}

%

\noindent
Note that totally bounded implies bounded.

\begin{proposition}
\label{feb28p3}If $S$ is a totally bounded subset of the q-normed space $X$,
and $f$ is a uniformly continuous mapping of $S$ into the q-normed space $Y$,
then $f(S)$ is totally bounded.
\end{proposition}

\begin{proof}
Given $\varepsilon>0$, let $\delta$ be a continuity associate for
$\varepsilon/2$ and $f$. Let $\left\{  x_{1},\ldots,x_{N}\right\}  $ be a
finitely enumerable $\delta$-approximation to $S$. For each $x\in S$ there
exists $k$ such that $\left\Vert x-x_{k}\right\Vert _{i}<\delta$ for all $i\in
I$. If $\left\Vert f(x)-f(x_{k})\right\Vert _{j}>\varepsilon/2$ for some $j\in
J$, then $\left\Vert x-x_{k}\right\Vert _{i}>\delta$ for some $i\in I$, which
is absurd. Hence $\left\Vert f(x)-f(x_{k})\right\Vert _{j}\leq\varepsilon
/2<\varepsilon$ for all $j\in J$. It follows that $\left\{  f(x_{1}%
),\ldots,f(x_{N})\right\}  $ is a finitely enumerable $\varepsilon
$-approximation to $f(S)$.
\end{proof}

\begin{corollary}
\label{feb28c1}If $u$ is a bounded linear mapping of the q-normed space $X$
into the q-normed space $Y$, and $S\subset X$ is totally bounded, then $u(S) $
is totally bounded in $Y$.
\end{corollary}

\begin{proof}
By \cite[Chapter 7, (5.6)]{BB}, $u$ is uniformly continuous on $X$ and hence
on $S$. It remains to apply Proposition \ref{feb28p3}.
\end{proof}

\begin{corollary}
\label{feb28c2}If $S$ is a totally bounded subset of a q-normed space $X$, and
$f$ is a uniformly continuous mapping of $S$ into $\mathbf{R}$, then
$\sup_{x\in S}f(x)$ and $\inf_{x\in S}f(x)$ exist.
\end{corollary}

\begin{proof}
This follows from Proposition \ref{feb28p3} and \cite[Proposition
2.2.5]{BVtech}.
\end{proof}

\begin{proposition}
\label{mar0326}The following are equivalent conditions on a q-normed space
$X.$

\begin{enumerate}
\item[\emph{(i)}] Every element of $X$ is normable.

\item[\emph{(ii)}] Every totally bounded subset of $X$ is located.
\end{enumerate}
\end{proposition}

\begin{proof}
Suppose that (i) holds, and let $K\subset X$ be totally bounded. Given
$\alpha,\beta\in\mathbf{R}$ with $\alpha<\beta$, let $S=\left\{  x_{1}%
,\ldots.x_{n}\right\}  $ be a finitely enumerable $\frac{1}{2}(\beta-\alpha
)$-approximation to $K$. For each $x\in X$, since $S$\ is located in $X$ (by
Proposition \ref{mar03p2}), either for each $k\leq n\ $there exists $i_{k}$
such that $\left\Vert x-x_{k}\right\Vert _{i_{k}}>\frac{1}{2}(\alpha+\beta)$
or else there exists $k\leq n$ such that $\left\Vert x-x_{k}\right\Vert
_{i}<\beta$ for all $i\in I$. In the former case, for each $y\in K$, choosing
$k\leq n$ such that $\left\Vert y-x_{k}\right\Vert _{i}<\frac{1}{2}%
(\beta-\alpha)$ for all $i\in I$, we have%
\[
\left\Vert x-y\right\Vert _{i_{k}}\geq\left\Vert x-x_{k}\right\Vert _{i_{k}%
}-\left\Vert y-x_{k}\right\Vert _{i_{k}}>\tfrac{1}{2}(\alpha+\beta)-\tfrac
{1}{2}(\beta-\alpha)=\alpha.
\]
Since $\alpha,\beta$, and $x$ are arbitrary, it follows that that $K$ is
located in $X$. Thus (i)\ $\Rightarrow\ $(ii). Conversely, if (ii) holds, then
every singleton subset of $X$ is located; whence (i) holds, by Proposition
\ref{mar03p2}.
\end{proof}

\begin{proposition}
\label{feb28p4}Let $K$ be a totally bounded subset of the q-normed space $X$,
and let $S\subset K$ be located in $K$. Then $S$ is totally bounded.
\end{proposition}

\begin{proof}
Given $\varepsilon>0$, construct a finitely enumerable $\varepsilon
/3$-approximation $\left\{  x_{1},\ldots,x_{N}\right\}  $ to $K$. Write
$\left\{  1,\ldots,N\right\}  $ as a union of two sets $P$ and $Q$, where

\begin{itemize}
\item[--] if $k\in P$, then there exists $s_{k}\in S$ such that $\left\Vert
x_{k}-s_{k}\right\Vert _{i}<2\varepsilon/3$ for all $i\in I$, and

\item[--] if $k\in Q$, then for each $s\in S$ there exists $i\in I$ such that
$\left\Vert x_{k}-s\right\Vert _{i}>\varepsilon/3$.
\end{itemize}

%

\noindent
Given $s\in S$, choose $k$ such that $\left\Vert s-x_{k}\right\Vert
_{i}<\varepsilon/3$ for all $i\in I$. Then $k\notin Q$, so $k\in P$ and
therefore%
\[
\left\Vert s-s_{k}\right\Vert _{i}\leq\left\Vert s-x_{k}\right\Vert
_{i}+\left\Vert x_{k}-s_{k}\right\Vert _{i}<\frac{\varepsilon}{3}%
+\frac{2\varepsilon}{3}=\varepsilon
\]
for all $i\in I$. Thus $\left\{  s_{k}:k\in P\right\}  $ is a finitely
enumerable $\varepsilon$-approximation to $S$.
\end{proof}

\section{Finite-dimensional q-normed spaces}

First place on record this theorem from \cite[Theorem 20.5]{OLG48}.

\begin{theorem}
\label{mar16t1}Every $n$-dimensional locally convex space over $\mathbf{R}$ is
homeomorphic to $\mathbf{R}^{n}$.
\end{theorem}

%

\noindent
In fact, the proof of \cite[Theorem 20.5]{OLG48} shows that if $\left(
X,\left(  p_{i}\right)  _{i\in I}\right)  $ is an $n$-dimensional real locally
convex space with basis $\left\{  b_{1},\ldots,b_{n}\right\}  $, then the
\textbf{canonical linear bijection}%
\begin{equation}
f_{X}:\sum_{i=1}^{n}\lambda_{i}b_{i}\rightsquigarrow\left(  \lambda_{1}%
,\ldots,\lambda_{n}\right)  \label{a12}%
\end{equation}
of $X$ onto $\mathbf{R}^{n}$ is uniformly continuous, and has uniformly
continuous inverse, relative to the locally convex structures on $X$ and
$\mathbf{R}^{n}$.\label{0000804}

\begin{corollary}
\label{mar31c1}Let $\left(  X,\left(  \left\Vert \ \right\Vert _{i}\right)
_{i\in I}\right)  $ be a $n$-dimensional real q-normed with basis $\left\{
b_{1},\ldots,b_{n}\right\}  $. Then $f_{X}\ $is bi-uniformly continuous
relative to the q-norms on the two spaces.
\end{corollary}

\begin{proof}
Since $f_{X}$ is bi-uniformly continuous in the locally-convex-space sense,
the result follows from Proposition \ref{mar31p1}.
\end{proof}

\begin{proposition}
\label{mar31p2}Let $\left(  X,\left(  \left\Vert \ \right\Vert _{i}\right)
_{i\in I}\right)  $ be a $n$-dimensional real q-normed with basis $\left\{
b_{1},\ldots,b_{n}\right\}  $. Then for each $j\leq n$ the $j^{\mathbf{th}}%
$\textbf{\ coordinate functional}%
\[
u_{j}:\sum_{k=1}^{n}\lambda_{k}b_{k}\rightsquigarrow\lambda_{j}%
\ \ (\mathbf{\lambda}\in\mathbf{R}^{n})
\]
is q-norm uniformly continuous on $X$.
\end{proposition}

\begin{proof}
The function $u_{j}$ is the composition of the q-norm uniformly continuous
mappings $f_{X}$ on $X\ $and $\mathbf{\lambda}\rightsquigarrow\lambda_{j}$ on
$\mathbf{R}^{n})$.
\end{proof}

%

\medskip

We say that a q-normed space $\left(  X,\left(  \left\Vert \ \right\Vert
_{i}\right)  _{i\in I}\right)  $ is

\begin{itemize}
\item[--] \textbf{locally totally bounded }if every bounded subset of $X$ is
contained in some totally bounded set, and

\item[--] \textbf{locally compact} if every bounded subset of $X$ is contained
in some \textbf{compact}---that is, totally bounded and q-complete---set.
\end{itemize}

\begin{proposition}
\label{mar31p3}Every finite-dimensional q-normed space is locally compact.
\end{proposition}

\begin{proof}
Let $\left(  X,\left(  \left\Vert \ \right\Vert _{i}\right)  _{i\in I}\right)
$ be an $n$-dimensional q-normed space, with basis $\left\{  b_{1}%
,\ldots,b_{n}\right\}  $, let $f_{X}$ be as at (\ref{a12}), and let $B$ be a
bounded subset of $X$. There exists $r>0$ such that $\left\Vert x\right\Vert
_{i}<r$ for all $x\in B$ and all $i\in I$. Also, there exists $c>0$ with the
property: for each $x\in X$ there exists $i\in I$ such that $c\left\Vert
x\right\Vert _{i}>\left\Vert f_{X}(x)\right\Vert -1$. Thus for each $x\in B$
there exists $i\in I$ such that $\left\vert f_{X}(x)\right\vert <c\left\Vert
x\right\Vert _{i}+1<cr+1$, and therefore $f_{X}(B)$ is bounded in
$\mathbf{R}^{n}$. Since the latter space is locally compact \cite[4.1.6]%
{BVtech}, there exists a compact $K\subset\mathbf{R}^{n}$ such that
$f_{X}(B)\subset K$. Then $B\subset f^{-1}(K)$. Since $f^{-1}$ is uniformly
continuous, $K$ is totally bounded in $X$, by Proposition \ref{feb28p3}.
Moreover, since $K$ is complete and $f_{X}^{-1}$ is bi-continuous, $f^{-1}(K)$
is complete and hence compact.
\end{proof}

%

\medskip

It is tempting to suggest that every element of a finite-dimensional q-normed
space is normable. Here is a Brouwerian counterexample. Let $\left(
a_{n}\right)  _{n\geq1}$ be a binary sequence with $a_{1}=0$. Define seminorms
$\left\Vert \ \right\Vert _{n}$ $(n\geq1)\ $on $\mathbf{R}$ as follows: for
each $x\in\mathbf{R}$,%
\[
\left\Vert x\right\Vert _{n}=\left\{
\begin{array}
[c]{ll}%
(1-a_{n})\left\vert x\right\vert  & \text{if }a_{n}=0\text{ or }a_{n-1}=1,\\
& \\
2\left\vert x\right\vert  & \text{if }a_{n}=1-a_{n-1}%
\end{array}
\right.
\]
Then $\left\Vert x\right\Vert _{n}\leq2\left\vert x\right\vert $ for all
$x\in\mathbf{R}$, so $\left(  \mathbf{R},\left(  \left\Vert \ \right\Vert
_{n}\right)  _{n\geq1}\right)  $ is a q-normed space. If%
\[
\left\Vert 1\right\Vert _{\infty}\equiv\sup\left\{  \left\Vert 1\right\Vert
_{n}:n\geq1\right\}
\]
exists, then either $1<\left\Vert 1\right\Vert _{\infty}$ or $\left\Vert
1\right\Vert _{\infty}<2$. In the first case, there exists $N$ such that
$\left\Vert 1\right\Vert _{n}>1$, so $a_{N}=1-a_{N-1}$ and therefore $a_{N}%
=1$. In the case $\left\Vert 1\right\Vert _{\infty}<2$, we must have $a_{n}=0$
for all $n$. Thus the statement

\begin{quote}
\emph{Every element of a finite-dimensional q-normed space is normable}
\end{quote}

%

\noindent
implies the essentially nonconstructive omniscience principle \textsf{LPO}:

\begin{quote}
For each binary sequence $\left(  a_{n}\right)  _{n\geq1}$ either $a_{n}=0$
for all $n$ or else there exists $N$ such that $a_{N}=1$
\end{quote}

%

\noindent
\textsf{Hence, a fortiori, the statement}

\begin{quote}
\emph{Every element of a finite-dimensional locally convex space is normable}
\end{quote}

%

\noindent
implies \textsf{LPO}.

Need the unit ball of a finite-dimensional q-normed space be totally bounded?
Consider the same q-normed space $\left(  \mathbf{R},\left(  \left\Vert
\ \right\Vert _{n}\right)  _{n\geq1}\right)  $ as in the preceding paragraph,
and suppose that its unit ball $B$ is totally bounded. Let $\left\{
x_{1},\ldots,x_{\nu}\right\}  $ be a finite $\frac{1}{4}$-approximation to $B$
relative to $\left(  \left\Vert \ \right\Vert _{n}\right)  _{n\geq1}$, and let
$\xi=\max_{1\leq k\leq\nu}x_{k}$. If there exists $N$ such that $a_{N}%
=1-a_{N-1}$, then $2\left\vert x_{k}\right\vert =\left\Vert x_{k}\right\Vert
_{N}\leq1$ for each $k\leq\nu$, so $\left\vert \xi\right\vert \leq\frac{1}{2}%
$. On the other hand, if $a_{n}=0$ for all $n$, then $1\in B$ and so there
exists $k\leq\nu$ such that $\left\vert 1-x_{k}\right\vert =\left\Vert
1-x_{k}\right\Vert _{n}<\frac{1}{4}$ for all $n$; whence $1-\xi\leq
1-x_{k}<\frac{1}{4}$ and therefore $\xi>\frac{3}{4}$. Now, either $\xi
>\frac{1}{2}$ or $\xi<\frac{3}{4}$. In the first case we cannot have $a_{N}=1$
for any $N$, so $a_{n}=0$ for all $n$; in the second, we cannot have $a_{n}=0$
for all $n$. Thus the statement\label{ere2}

\begin{quote}
\emph{The unit ball of a finite-dimensional q-normed space is totally bounded}
\end{quote}

%

\noindent
implies the omniscience principle \textsf{WLPO}:

\begin{quote}
For each binary sequence $\left(  a_{n}\right)  _{n\geq1}$ either $a_{n}=0$
for all $n$ or it is not the case that $a_{n}=0$ for all $n$.
\end{quote}

\section{For the future}

Perhaps the most obvious examples of q-normed spaces for deeper study are the
dual q-normed space $X^{\ast}\ $and its dual $X^{\ast\ast}$, where $\left(
X,\left(  \left\Vert \ \right\Vert _{i}\right)  _{i\in I}\right)  $ is a given
q-normed (but not necessarily either normed or uniformly convex) space. In
order to carry out such deeper study, it seems that we may need a version of
the Hahn-Banach theorem applicable to q-normed spaces (cf. \cite{Ishihara2}).
This is something to pursue in\ future research.%

\bigskip

\section{Disclaimer}

No longer having a copy of \cite{Johns}, and being unable to trace it online
(even in the University of Liverpool Library), in this article I may well have
re-proved some results originally due to Johns. However, there is definitely
original work of mine in the foregoing.%

\bigskip
%

\bigskip

\vfill

\begin{flushright}
\texttt{{\small dsbridges 240626}}
\end{flushright}


\begin{thebibliography}{99}                                                                                               %


\bibitem {Bishop}Errett Bishop, \emph{Foundations of Constructive Analysis},
McGraw-Hill, New York, 1967.

\bibitem {BB}E. Bishop and D.S. Bridges, \emph{Constructive Analysis},
Grundlehren der math. Wissenschaften \textbf{279}, Springer Verlag,
Heidelberg-Berlin-New York, 1985.

\bibitem {dsbconvex}D.S. Bridges, `Constructive Notes on Locally Convex
Spaces', 2026, online at
{\small https://workdrive.zoho.com/file/8z28q370440398d2549c38e8101224a1997e9}

\bibitem {BVtech}D.S. Bridges, L.S. V\^{\i}\c{t}\u{a}, \emph{Techniques of
Constructive Analysis}, Universitext, Springer New York, 2006.

\bibitem {OLG48}D.S. Bridges, L.S. V\^{\i}\c{t}\u{a}: `The constructive
uniqueness of the locally convex topology' in: \emph{From Sets and Types to
Topology and Analysis} (L. Crosilla and P. Schuster, eds), 304--315, Oxford
Logic Guides \textbf{48}, Clarendon Press, Oxford, 2005.

\bibitem {BIMc13}D.S. Bridges, H. Ishihara, and M. McKubre-Jordens, `Uniformly
convex Banach spaces are reflexive---constructively', Math. Logic. Quart.
\textbf{59}(4--5), 352--356, 2013.

\bibitem {Handbook}D.S. Bridges, H. Ishihara, M.J. Rathjen, H. Schwichtenberg
(editors),\emph{\ Handbook of Constructive Mathematics}, Encyclopedia of
Mathematics and Its Applications \textbf{185}, Cambridge University Press, 2023.

\bibitem {Ishihara2}H. Ishihara, `The constructive Hahn-Banach theorem,
revisited', in: \emph{Contemporary Logic and Computing }(A. Rezu\c{s}, ed),
638--663, Landscapes in Logic \textbf{1}, College Publications, Rickmansworth,
U.K., 2020.

\bibitem {Ishihara}H. Ishihara, `On the constructive Hahn-Banach theorem',
Bull. London. Math. Soc. \textbf{21}, 79--81, 1989.

\bibitem {Johns}D.L. Johns, \emph{A Constructive Approach to Duality and
Orlicz Spaces}, Ph.D. thesis, University of Liverpool, 1977.
\end{thebibliography}
\end{document}